\documentclass[12pt]{amsart}

\usepackage{amssymb}
\usepackage{latexsym}
\usepackage{exscale}
\usepackage{graphicx}
\usepackage{mathrsfs}
\usepackage[colorlinks=true, pdfstartview=FitV, linkcolor=blue,
citecolor=blue, urlcolor=blue]{hyperref}

\calclayout

\allowdisplaybreaks

\numberwithin{equation}{section}

\theoremstyle{plain}
\newtheorem{theorem}[equation]{Theorem}
\newtheorem{lemma}[equation]{Lemma}

\theoremstyle{definition}
\newtheorem{defi}[equation]{Definition}

\theoremstyle{remark}
\newtheorem{remark}[equation]{Remark}

\newcommand{\supp}{\mathrm{supp}}

\DeclareMathOperator*{\essinf}{ess\,inf}

\def\Xint#1{\mathchoice
  {\XXint\displaystyle\textstyle{#1}}%
  {\XXint\textstyle\scriptstyle{#1}}%
  {\XXint\scriptstyle\scriptscriptstyle{#1}}%
  {\XXint\scriptscriptstyle\scriptscriptstyle{#1}}%
  \!\int}
\def\XXint#1#2#3{{\setbox0=\hbox{$#1{#2#3}{\int}$}
    \vcenter{\hbox{$#2#3$}}\kern-.5\wd0}}

\def\avgint{\Xint-}

\title[Necessary conditions]{Neccessary conditions for two weight inequalities for singular integral operators}

\author{David Cruz-Uribe, OFS}
\author{John-Oliver MacLellan}
\address{David Cruz-Uribe, OFS \\ 
Department of Mathematics \\ 
University of Alabama \\
Tus\-ca\-loosa, AL 35487, USA}
\email{dcruzuribe@ua.edu}
\address{John-Oliver MacLellan, Department of Mathematics, University of Alabama\\
Tus\-ca\-loosa, AL 35487, USA}
\email{jomaclellan@crimson.ua.edu}

\dedicatory{In memoriam of Benjamin Muckenhoupt and Richard Wheeden,\\
  who pioneered the study of two weight inequalities.}

\thanks{The first author is supported by 
  research funds from the Dean of the College of Arts \& Sciences,
  University of Alabama. }



\date{04/18/2020}

\begin{document}

\begin{abstract}
 We prove necessary conditions on pairs of measures $(\mu,\nu)$ for a
 singular integral operator $T$ to satisfy weak $(p,p)$ inequalities,
 $1\leq p<\infty$, provided the kernel of $T$ satisfies a weak
 non-degeneracy condition first introduced by Stein~\cite{MR1232192},
 and the measure $\mu$ satisfies a weak doubling condition related to
 the non-degeneracy of the kernel.  
 We also show similar results for pairs of measures $(\mu,\sigma)$ for
 the operator $T_\sigma f = T(f\,d\sigma)$, which has come to play an
 important role in the study of weighted norm inequalities.  Our major
 tool is a careful analysis of the strong type inequalities for
 averaging operators; these results are of interest in their own
 right.  Finally, as an application of our techniques, we show that in
 general a singular operator does not satisfy the endpoint strong type
 inequality $T : L^1(\nu) \rightarrow L^1(\mu)$.   Our results unify
 and extend a number of known results.
\end{abstract}

\maketitle

\section{Introduction} 
The goal of this paper is to establish necessary conditions for two
weight, weak type inequalities for Calder\'on-Zygmund operators. This
problem has a long history.  In the one weight case it is well known
that if each of the Riesz transforms is of weak type $(p,p)$ with
respect to a weight $w$, then $w \in A_p$. See for example
\cite[~Theorem~3.7, p.~417]{MR807149}. Stein \cite[p.~210]{MR1232192},
showed that if any convolution type singular integral operator whose
kernel satisfies a weak non-degeneracy condition is bounded on
$L^p(w)$, then $w \in A_p$. The necessity of two weight $A_p$ for the
weak $(p,p)$ inequality for the Hilbert transform was established by
Muckenhoupt and Wheeden~\cite{MR0417671}.

In this paper we consider two versions of this problem. First suppose $(\mu, \nu)$
is a pair of positive regular Borel measures on $\mathbb{R}^n$, where
$\mu$ satisfies a weak doubling condition (see Definition
\ref{defi:directionaldoubling}) and $T$ is a Calder\'on-Zygmund operator
whose kernel satisfies a weak non-degeneracy condition (see Definition
\ref{defi:nondegeneracy}). We first prove the following result on the
necessity of the two-weight $A_p$ condition (see
Definition \ref{def:A_p}).

\begin {theorem}\label{thm:necessity measures}
  Let $T$ be a Calder\'on-Zygmund operator with a non-degenerate
  kernel in the direction $u_0$. Suppose that for some
  $ 1 \leq p < \infty$, and a pair of positive regular Borel measures
  $(\mu,\nu)$, with $\mu$ directionally doubling in the direction
  $u_0$,
\begin{equation}
\|Tf\|_{L^{p,\infty}(\mu)} \leq C \|f\|_{L^p(\nu)}.
\end{equation}
Then:
\begin{enumerate}
\item $d\nu=d\nu_s+vdx$ where $v \in L^1_{\text{loc}}$ and $\nu_s$ is singular;
\item $\mu \ll \nu $, and $\mu \ll dx$, so $d\mu=u dx$ where $u\in L^1_{\text{loc}}$;
\item  $(u,v) \in A_p$ and $u(x) \leq Cv(x)$ a.e.
\end{enumerate} 
\end{theorem}

The second version of the problem is to let $( \mu,\sigma)$ be a pair
of positive regular Borel measures on $\mathbb{R}^n$, and consider the
singular integral operator $T_\sigma$ defined by
$$T_\sigma f(x)= T( f\,d\sigma)(x) = \int_{\mathbb{R}^n} K(x,y) f(y) \; d\sigma(y).$$
This approach to weighted norm inequalities first appeared implicitly in \cite{MR676801} for the maximal
operator. We establish necessary conditions on $(\mu,\sigma)$ for
$T_\sigma$ to satisfy the weak type inequality
$$ \mu(\{ x: |T_\sigma f(x)| > \lambda \}) \leq \frac{C}{\lambda^p} \int |f|^p d\sigma $$
for $1<p<\infty$.  (See Definition \ref{def:truncated weak (p,p)} for
a more careful definition of this operator and the meaning of this
inequality.)  This problem was considered in \cite{MR2957550}, where
they proved the necessity of the $A_p$ condition for measures (see
Definition \ref{def: Ap condition for measures}) assuming a strong
ellipticity condition.  More precisely, they assumed that there exists a family of kernels
$\{K_j\}_{j=1}^N$  such that given any unit direction vector $u$,
there exists $j$ such that $K_j$
satisfies a  non-degeneracy condition in the direction $u$ (see
Definition \ref{defi:nondegeneracy}).  In \cite{MR3470665}, they were
able to prove the stronger $PA_p$ condition (see Definition
\ref{def:Poisson Ap}) is necessary with a similar hypothesis assuming a strong
$(p,p)$ inequality. Our results are similar but we only assume that we
have a single operator 
$T_{\sigma}$ whose kernel is non-degenerate in one direction. We obtain
the necessity of $A_p$ under the additional hypothesis that $\mu$
satisfies a weak doubling condition related to the non-degeneracy
condition (see Defintion
\ref{defi:directionaldoubling}). We prove this result for completeness
since it is a simple 
application of the techniques used to prove Theorem \ref{thm:necessity
  measures}. 

\begin{theorem}\label{thm:necessity Tsigma}
  Let $(\mu,\sigma)$ be a pair of positive regular Borel measures with
  $\mu$ directionally doubling in the direction $u_0$. Suppose the
  operator $T_\sigma$ has a non-degenerate kernel in the direction
  $u_0$, and that for some $ 1 < p < \infty$,
\begin{equation}\label{weak(p,p) Tsigma}
\|T_\sigma f\|_{L^{p,\infty}(\mu)} \leq C \|f\|_{L^p(\sigma)}.
\end{equation}
Then $(\mu, \sigma) \in A_p$.
\end{theorem}

More importantly, we also establish the necessity of the $PA_p$
condition with the additional assumption that $\sigma$ is doubling.

\begin{theorem}\label{thm:necessity Tailed Ap}
  Let $(\mu, \sigma)$ be a pair of positive regular Borel
  measures with $\mu$  directionaly doubling in the
  direction $u_0$ and $\sigma$ doubling. Suppose $T_\sigma$ has a
  non-degenerate kernel in the direction $u_0$ and for some
  $1<p<\infty$,
\begin{equation}
\|T_\sigma f\|_{L^{p,\infty}(\mu)} \leq C \|f\|_{L^p(\sigma)}.
\end{equation}
Then $(\mu, \sigma) \in PA_p$. 
\end{theorem}

Finally as an application of our techniques we consider the question
of whether a Calder\'on-Zygmund operator can be bounded from
$L^1(\mu)$ to $L^1(\nu)$ for a pair of positive Borel measures
$(\mu, \nu)$.  If the operator under consideration is translation
invariant and the measures $(\mu,\nu)$ are regular, it is known that
this is impossible. See \cite[p.~468]{MR807149}. In \cite[~Theorem ~4]
{MR0417671} Muckenhoupt and Wheeden derived a necessary condition for the Hilbert
transform to be bounded from $L^1(v)$ to $L^1(u)$, where $u$ and $v$
are weights. We obtain an
analogous estimate (see \eqref{eq:strong 1-1 3} in the proof of
Theorem \ref{thm:strong 1-1}), but give a more complete
characterization in terms of measures.

\begin{theorem}\label{thm:strong 1-1}
  Let $T$ be a Calder\'on Zygmund operator with a non-degenerate
  kernel in the direction $u_0$, and let $(\mu,\nu)$ be positive Borel
  measures on $\mathbb{R}^n$.
\begin{enumerate}

\item If $\nu$ is singular with respect to Lebesgue measure and
  $T:L^1(\nu) \rightarrow L^1(\mu)$, then $\mu= 0$.

\item If $\mu$ is a regular measure with $d\mu=d\mu_s + udx$ where $\mu_s$ is
  singular with respect to Lebesgue measure, $u \not\equiv 0$, 
  $d\nu= d\nu_s + vdx$ where $\nu_s$ is singular with respect to
  Lebesgue measure, and $v$ is a non-negative measurable function such
  that $v(x)< \infty$ a.e, then $T$ is not bounded from $L^1(\nu)$ to
  $L^1(\mu)$.

\item If $\mu$ is a  regular measure that is singular
  with respect to Lebesgue measure, and directionaly doubling in the
  direction $u_0$, and $\nu$ is a positive regular Borel measure, then
  $T$ is not bounded from $L^1(\nu)$ to $L^1(\mu)$.
\end{enumerate}
\end{theorem}

\begin{remark}
  The following example shows that the hypothesis in $(2)$ that $v(x)< \infty$
  a.e.~is needed.  Let $d\mu=\chi_{[-1,1]} \; dx$,
  $d\nu= \chi_{\mathbb{R} \setminus [-2,2]} \; dx+ \infty \cdot
  \chi_{[-2,2]}\; dx$, and let $Tf(x)=Hf(x)$. Then $(\mu,\nu)$
  satisfies the key estimate $\eqref{eq:strong 1-1 3}$ below, and for $f$ with
  $\text{supp}(f) \subset \{x:|x|>2\}$ we have that
\begin{multline*}
\int_{\mathbb{R}} |Hf(x)| \; d\mu(x) = \int_{-1}^1 |Hf(x)| \; dx 
\leq \int_{-1}^1 \int_{|y|>2} \frac{|f(y)|}{|x-y|} \; dy \; dx \\
= \int_{|y|>2} \int_{-1}^1   \frac{1}{|x-y|} \; dx \; |f(y)| \; dy 
 \leq 2 \int_{|y|>2} |f(y)| \; dy   = 2 \int_{\mathbb{R}} |f(y)| \; d\nu(y) .  
\end{multline*}
\end{remark}

\begin{remark}
  The following example shows that the hypothesis in $(2)$ that $\mu$ is not
  totally singular with respect to Lebesgue measure is needed. Let
  $\mu=\delta(0)$, $\nu= \frac{1}{x} dx$ and let $Tf(x)=Hf(x)$. Then
  for any $f \in L^1(\nu)$,
$$\int_{\mathbb{R}} |Hf(x)| d\mu(x)= |Hf(0)|=  \left|  \int_{\mathbb{R}} \frac{f(y)}{y} \; dy \right | \leq \int_\mathbb{R} |f(y)| d\nu(y).$$
\end{remark}

\medskip

The main idea in our proofs is to reduce the problem of obtaining
necessary conditions for the $L^p$ boundedness of singular integrals
to that of averaging operators (see Definitions~\ref{def: averaging
  operator} and~\ref{def:averaging operator Asigma}). For
singular integrals $T$, we work with the averaging operator $A_Q$.  When
$\mu=udx$, $\nu= vdx$ it is well known that the $A_p$ condition
characterizes the strong type inequality for $A_Q$;  see
Jawerth~\cite{MR833361}.  For completeness we prove this result (see Theorem
\ref{thm:Ap neccessity averaging operators}).  Furthermore,  we
obtain a characterization of the strong type inequality for $A_Q$ when
$(\mu, \nu)$ are positive regular Borel measures (see Theorem
\ref{thm: necessity averaging operators measures}). This result is new
and is interesting in its own right.

To study the singular integrals $T_\sigma$ we introduce the analogous averaging
operator $A_{Q,\sigma}f = A_Q(f\,d\sigma)$.  We show that  the $A_p$
condition for measures is also necessary and
sufficient (see Theorem \ref{thm: Neccessity Asigma}) for these
operators to be bounded, $1<p<\infty$. 

The rest of the paper is organized as follows. In Section 2 we give
preliminary definitions and notation used in this paper. In Section 3
we prove Theorem \ref{thm:necessity measures}. In Section 4 we prove
Theorem \ref{thm:necessity Tsigma} and Theorem \ref{thm:necessity
  Tailed Ap}.  Finally, in Section 5 we prove Theorem \ref{thm:strong 1-1}.

\section{Preliminaries}
Throughout this paper will use the following notation. The symbol $n$
will denote the dimension of the Euclidean space
$\mathbb{R}^n$. $Q(x,r)$ denotes the cube with center
$x \in \mathbb{R}^n$ and sidelength $2r$, while $B(x,r)$ denotes the
ball with center $x \in \mathbb{R}^n$ and radius $r$. For a cube $Q$,
$rQ$ is the cube with the same center as $Q$ and with side length $r$
times the length of $Q$. Positive constants $C,c$ may change value at
each appearance. Sometimes we will indicate the dependence on certain
parameters by writing for instance, $C(n,p)$ etc. We will work
extensively with average integrals and use the notation,
$$ \avgint_Q u \; dx = \ \frac{1}{|Q|} \int_Q u \; dx.$$

We now define the singular integral operators we are interested in.
For further, details, see~\cite{MR1800316}. 

\begin{defi}
  We say that an operator $T$ defined on measurable functions is a
  Calder\'on-Zygmund operator if $T$ is bounded on $L^2(\mathbb{R}^n)$
  and for any $ f\in L^2_c(\mathbb{R}^n)$ we have the representation
\begin{equation*}
Tf(x)= \int_{\mathbb{R}^n} K(x,y)f(y) \; dy, \quad x \notin  \text{supp}(f).
\end{equation*}
Here $K(x,y)$ is a kernel defined for all $ x \neq y$ in $ \mathbb{R}^n \times \mathbb{R}^n$, that satisfies the standard estimates
\begin{equation}\label{eq:size}
|K(x,y)| \leq \frac{C_0}{|x-y|^n}
\end{equation} and
\begin{equation} \label{eq:smoothness}
|K(x+h,y)- K(x,y)| + |K(x,y+h) - K(x,y)| \leq C_0\frac{|h|^ \delta }{|x-y|^{n+\delta}}
\end{equation}
for all $|h|< \frac{1}{2} |x-y|$ and some fixed $\delta >0$.
\end{defi}

 \medskip

 We want to define the operators $T_\sigma$  more carefully;  to do so
 we 
 follow the treatment given in~\cite{MR3688149}.  Let $(\mu,\sigma)$
 be a pair of regular Borel measures.   Fix a Calder\'on-Zygmund
 operator $T$ with kernel $K$.  
Let
 $\{\eta_{\epsilon, R}\}_{0< \epsilon< R <\infty}$ be a family of
 non-negative truncation functions with supports in the annuli
 $\epsilon<|x|<R$, and such that $\eta_{\epsilon,R}(x)=1$ if $2\epsilon<|x|<\frac{R}{2}$.  For example, we
 can take $\eta_{\epsilon, R}= \chi_{\{\epsilon<|x|<R\}}$, but other
 choices are possible.    Define the family of truncated kernels
 ${K_{\epsilon,R} (x,y) = \eta_{\epsilon,R}(x-y)K(x,y)}$.  These are bounded
 with compact support for a fixed $x$ or $y$.  Thus, the truncated
 operators defined by
\begin{equation*}
T^{\epsilon, R}_{\sigma} f(x)= \int_{\mathbb{R}^n} K_{\epsilon,R}(x,y) f(y) \;  d\sigma(y), \quad x\in \mathbb{R}^n,
\end{equation*} 
are pointwise well defined for $f\in L^1_{\text{loc}}$.   Hereafter,
we will assume that each of the  truncated kernels
$\{K_{\epsilon,R}\}_{0< \epsilon< R <\infty}$ satisfies the standard
kernel estimates~\eqref{eq:size} and~\eqref{eq:smoothness} with
uniform constants.

\begin{defi}\label{def:truncated weak (p,p)}
Given a Calder\'on-Zygmund operator $T$ with kernel $K$, 
we say that $T_\sigma$ satisfies the weak $(p,p)$ inequality,
$1<p<\infty$, provided that there exists a  family of truncations $\{\eta_{\epsilon, R}\}_{0<
  \epsilon< R <\infty}$ such that for all $f\in L^p(\sigma)$, 
\begin{equation}\label{eq:truncated weak (p,p)}
\|T_\sigma^{\epsilon,R} f \|_{L^{p,\infty}(\mu)} \leq C\|f\|_{L^p(\sigma)} 
\end{equation}
with constant independent of $\epsilon$ and $R$.  In this case we write
\begin{equation*}
\|T_\sigma f\|_{L^{p,\infty}(\mu)} \leq C\|f\|_{L^p(\sigma)}.
\end{equation*}
\end{defi}

Given this definition, in our proofs below we will need to fix
particular values of $\epsilon$ and $R$ and apply
inequality~\eqref{eq:truncated weak (p,p)}.  We will, however, generally
write $T_\sigma$ instead of $T^{\epsilon, R}_{\sigma}$ when there is
no possibility of confusion.

\begin{remark}
  While we need to fix a family of truncations to define $T_\sigma$,
  the choice is less important than it might seem at first.  In
  \cite{MR3688149}, they showed that if the pair $(\mu,\sigma)$
  satisfies the $A_p$ condition for measures, \eqref{def: Ap condition for
    measures} below, then the corresponding strong $(2,2)$ inequality for
  $T_\sigma$ holds independent of the choice of truncations used.
\end{remark}

\begin{defi}\label{defi:nondegeneracy}
  Given a Calder\'on-Zygmund operator $T$ with kernel $K(x,y)$, we say
  $T$ has a non-degenerate kernel if there exists $a>0$, and a unit
  vector $u_0$ such that for $x,\,y \in \mathbb{R}^n$,  $x-y = t u_0$, $t\in\mathbb{R}$,
\begin{equation}\label{eq:nondegeneracy}
|K(x,y)| \geq \frac{a}{|x-y|^n}.
\end{equation}
\end{defi}

For example, \eqref{eq:nondegeneracy} holds for the Hilbert transform
as well as of any of the Riesz transforms in the direction $e_j$.
However, not all singular integrals satisfy this property. See for
example \cite[Lemma 1.4]{MR1376747} where they construct 
a ``one-sided'' Calder\'on-Zygmund kernel with support in $(0,
\infty)$; they establish that a sufficient condition for this operator
to be bounded is a ``one-sided'' $A_p$ condition that is strictly
weaker than the conditions we consider.

\medskip

\begin{defi}\label{def:A_p}
Let $u,\,v$ be non-negative, measurable functions. We say the pair
$(u,v) \in A_p$, $1<p< \infty$, if
\begin{equation}\label{eq:A_p}
[u,v]_{A_p}= \sup_{Q} \left( \avgint_Q u  \; dx \right) \left(  \avgint_Q v^{1-p'} \; dx \right)^{p-1} < \infty,
\end{equation}
and in $A_1$ if 
\begin{equation}\label{eq:A1}
\avgint_Q u \; dx \leq [u,v]_{A_1}  \essinf_{x \in Q} v(x).
\end{equation} 
\end{defi}

\begin{defi}\label{def: Ap condition for measures}
 If $(\mu, \sigma)$ are positive Borel measures, we say that $(\mu,
 \sigma) \in  A_p$, $ 1 < p<\infty$, if
\begin{equation}\label{eq:Ap condition for measures}
  [\mu, \sigma]_{A_p}= \sup_Q \frac{\mu(Q)}{|Q|} \left(\frac{\sigma(Q)}{|Q|} \right)^{p-1} < \infty.
\end{equation} 
\end{defi}
\begin{remark}
If $d\mu= u \; dx$, $d\sigma=v^{1-p'}\; dx$, then \eqref{eq:Ap condition for measures} is equivalent to \eqref{eq:A_p}. 
\end{remark}
\begin{remark}\label{remark: Ap balls cubes}
  It is straightforward to see that if \eqref{eq:A_p}, \eqref{eq:A1},
  or~\eqref{eq:Ap condition for measures} hold for any cube
  $Q \subset \mathbb{R}^n$, then they also hold for any ball
  $ B \subset \mathbb{R}^n$.  We will use this fact below.
\end{remark}

\begin{remark}
  Inequality~\eqref{eq:Ap condition for measures} implies that $\mu, \sigma$ do
  not share a common point mass: if there exists a point $a$ such that
  $\sigma \{a\} \mu \{a\}>0$, then the expression in \eqref{eq:Ap
    condition for measures} blows up as $Q$ shrinks to $\{a\}$.
\end{remark}

\begin{defi}\label{def:Poisson Ap}
We say the pair $(\mu,\sigma)$ is in $PA_p$, $1<p< \infty$ if for any cube $Q(y_0,r)$, 
\begin{equation}\label{eq:Poisson Ap}
\left(\frac{\mu(Q(y_0,r))}{|Q(y_0,r)|}\right) \left(\int_{\mathbb{R}^n} \left( \frac{r^{p'-1}}{(|x-y_0|+r)^{p'}} \right)^n d\sigma(x) \right)^{p-1} \leq C.
\end{equation}
\end{defi}

This condition first appeared in \cite{MR0417671} in one dimension
where they proved that it was necessary for the strong type inequality
for the Hilbert transform to hold. The $n$-dimensional version first
appeared in \cite{MR1175693} in the context of the fractional integral
operator. When $p=2$, this condition is sometimes called ``Poisson
$A_p$". This is because the second term on the left-hand side of
\eqref{eq:Poisson Ap} is approximately the Poisson extension of
$\sigma$ evaluated at a point in the upper half plane given
by $y_0$ and $r$.  It is straightforward to see that the $PA_p$ condition implies the $A_p$ condition.

\begin{defi}
  A positive measure $\mu$ is said to be doubling if there exists a
  constant $ C>0$ such that for any cube $ Q$,
  $\mu(2Q) \leq C \mu(Q)$.
\end{defi}

Equivalently $\mu$ is doubling if $\mu(P) \leq C \mu(Q)$, whenever
$P,Q$ are adjacent cubes with $|Q|=|P|$. (We say the cubes $P,Q$ are adjacent if the
boundaries of $P$ and $Q$ share a point in common.)

For our results we do not need to assume the full doubling condition,
but rather a ``directional'' doubling condition.

\begin{defi}\label{defi:directionaldoubling}
  Let $\mu$ be a positive Borel measure, and fix a unit vector
  $u_0$. We say $\mu$ is directionally doubling in direction $u_0$ if
  there exists a constant $C_\mu>0$ such that given adjacent cubes
  $P(x_0,r),\,Q(y_0,r)$ whose centres satisfy $x_0-y_0=tu_0$,
  $t \in \mathbb{R}$,
\begin{equation}\label{eq:directionaldoubling}
\mu(P(x_0,r)) \leq C_\mu \mu(Q(y_0,r)). 
\end{equation}
\end{defi}

\begin{remark}
  Definition \ref{defi:directionaldoubling} is weaker than the
  doubling condition. For example, for $ E \subset \mathbb{R}^2$,
  define $\mu(E)=\iint_E e^{-|x|} \; dx dy$. Then it is
  straightforward to show that $\mu$ is
  directionally doubling in the direction $e_2$ but is not doubling.
\end{remark}

\section{Proof of Theorem \ref{thm:necessity measures}}
In order to proceed with the proofs of our first main result we will
need to prove some preliminary results about averaging operators.

\begin{defi}\label{def: averaging operator}
Given a cube $Q$, define the averaging operator $A_Q$ on a function $ f \in L^1_{\text{loc}}$ by
$$ A_Q f(x)= \avgint_Q f(y) \; dy \;\chi_Q(x).$$
\end{defi}

The following result first appeared in \cite{MR833361} but to the best
of our knowledge a proof does not appear in the literature.  For
completeness we sketch the proof.

\begin{theorem}\label{thm:Ap neccessity averaging operators}
Given  a cube $Q$, $1 \leq p < \infty$,  and $(u,v) \in A_p$, for all $f \in L^p(v)$
\begin{equation*} 
 \|A_Q f \|_{L^p(u)} \leq [u,v]_{A_p}^{1/p} \|f\|_{L^p(v)}.
\end{equation*}
Conversely, given $1\leq p<\infty$, if $(u,v)$ are a pair of weights such that for every cube $Q$,
\begin{equation}\label{eq:Ap neccessity averaging operators 1 }
\|A_Q f \|_{L^p(u)} \leq K \|f\|_{L^p(v)},
\end{equation}
then $(u,v) \in A_p$. Moreover, $[u,v]_{A_p} \leq K^p $. 
\end{theorem}
\begin{proof}
Let $Q \subset \mathbb{R}^n$.  We first prove the sufficiency of the $A_1$ condition when $p=1$. Indeed,
\begin{equation*}
 \|A_Q f\|_{L^1(u)}= \int_{\mathbb{R}^n} \left\lvert \avgint_Q f \; dy \chi_Q(x) \right\rvert u \; dx  
 \leq  \int_Q \frac{u(Q)}{|Q|} |f| \; dy \leq [u,v]_{A_1} \int_{\mathbb{R}^n} |f|v \; dy. 
\end{equation*}

If $p>1$, by H\"older's inequality,
\begin{align*}
\|A_Q\|_{L^p(u)}^p &= \int_{\mathbb{R}^n} \left\lvert \avgint_Q f \; dy \; \chi_Q(x) \right\rvert ^p u\; dx\\
& \leq \left( \avgint_Q |f|v^\frac{1}{p} v^\frac{-1}{p} \; dy \right)^p u(Q) \\
&\leq \left( \int_Q |f|^p v \; dy \right) \left(\avgint_Q u \; dy \right) \left(\avgint_Q v^{1-p'} \; dy \right)^{p-1}\\
& \leq [u,v]_{A_p} \int_{\mathbb{R}^n} |f|^p v \; dy.
\end{align*}
To prove necessity, let  $S \subset Q$ be measurable and set $ f= \chi_S$. Then \eqref{eq:Ap neccessity averaging operators 1 } becomes
\begin{equation}\label{eq:Ap neccessity averaging operators 2}
u(Q) \left(\frac{|S|}{|Q|} \right)^p \leq Kv(S).
\end{equation} 
In \cite[p.~388]{MR807149} they show that if \eqref{eq:Ap neccessity averaging operators 2} holds, then $(u,v) \in A_p$, and that $[u,v]_{A_p} \leq K^p$. This completes the proof.   
\end{proof}

We now prove an analogue of Theorem \ref{thm:Ap neccessity averaging operators} for measures. 

\begin{theorem}\label{thm: necessity averaging operators measures}
  Let $(\mu,\nu)$ be a pair of positive regular Borel measures. Given
  $1 \leq p < \infty$, suppose that there exists a constant $C$ such
  that for every cube $Q$
\begin{equation} \label{eq:necessity averaging operators measures 1}
 \|A_Qf \|_{L^p(\mu)} \leq C \|f\|_{L^p(\nu)}.
\end{equation}
 
Then:
\begin{enumerate}
\item $d\nu=d\nu_s+vdx$ where $v \in L^1_{\text{loc}}$ and $\nu_s$ is singular;
\item $\mu \ll \nu $, and $\mu \ll dx$, so $d\mu=u \,dx$ where $u\in L^1_{\text{loc}}$;
\item  $(u,v) \in A_p$ and $u(x) \leq Cv(x)$ a.e.
\end{enumerate} 
 
\end{theorem}
\begin{proof}
Fix a cube $Q \subset \mathbb{R}^n$ and let $ S \subset Q$ be measurable. Let $f=\chi_S$ in \eqref{eq:necessity averaging operators measures 1}; then arguing as before we obtain
\begin{equation}\label{eq:necessity averaging operators measures 2}
\left(\frac{|S|}{|Q|} \right)^p \mu(Q) \leq C \nu(S).
\end{equation}

Suppose $\nu$ were singular with respect to Lebesgue measure and
$|S|>0$. Then there exists a set $A \subset \mathbb{R}^n$ such that
$|A|=0$ and $\nu(S)= \nu(S \cap A)$. If we replace $S$ with
$ S \setminus A$ in \eqref{eq:necessity averaging operators measures
  2}, we find $ \mu(Q)=0$ for any cube $Q \supset S$. This implies
that $\mu=0$.  Hence, $ d\nu= d\nu_s + v \; dx$ where
$v \in L^1_{\text{loc}}$ $v \neq 0$, and $\nu_s$ is singular.

Now fix any set $S \subset Q$ with $ \nu(S)=0$. Since $\nu$ is
regular, for any $\epsilon>0$ there exists an open set $E \supset S$
such that $\nu(E) < \epsilon$. Since $E$ is open, $ E = \cup_j Q_j$
where $ \{Q_j\}$ is a disjoint collection of dyadic cubes. If we let
$Q=S=Q_j$ in \eqref{eq:necessity averaging operators measures 2}, we
have
$$ \mu(S) \leq  \mu(E)= \sum_j  \mu(Q_j) \leq C \sum_j \nu(Q_j)= C \nu(E) < C \epsilon.$$
Since $\epsilon>0$ was arbitrary we have $\mu(S)=0$, and so $ \mu \ll \nu$. 

	We can now write \eqref{eq:necessity averaging operators measures 1} as
\begin{equation}
\mu(Q) \left\lvert \avgint_Q f \; dx \right\rvert^p \leq C \left( \int_{\mathbb{R}^n} |f|^p \; d\nu_s + \int_{\mathbb{R}^n} |f|^p v \; dx \right).
\end{equation}
Let $A= \supp( \nu_s)$. Since $|A|=0$, if we set $ f = \chi_{S \setminus A}$, we have\\
 
\begin{equation}\label{eq:necessity averaging operators measures 3}
\left( \frac{|S|}{|Q|} \right)^p \mu(Q) \leq C v(S).
\end{equation}
Using the same argument that showed $\mu \ll \nu$, replacing $\nu$
with $v$, we can see that $ \mu \ll vdx \ll dx$. Hence, $d\mu=u dx$ for
some $u\in L^1_{\text{loc}}$. Let $S=Q$ in \eqref{eq:necessity
  averaging operators measures 3}, then by the Lebesgue
differentiation theorem we have that $u(x) \leq C v(x)$ a.e.
Moreover, by \eqref{eq:necessity averaging operators measures 3} we
have that
$$ \left(\frac{|S|}{|Q|} \right)^p u(Q) \leq C v(S).$$
Hence, the fact that $(u,v) \in A_p$ follows as in the proof of Theorem \ref{thm:Ap neccessity averaging operators}.
\end{proof}

\begin{proof}[Proof of Theorem \ref{thm:necessity measures}.]
  We will show that given any cube $Q$ the averaging operator
  satisfies $A_{Q}: L^p(\nu) \rightarrow L^p(\mu)$. The desired
  conclusion then follows from Theorem \ref{thm: necessity averaging
    operators measures}. Choose a constant $t \geq 4$ such that
  $2C_0(1+2^{n+ \delta})t^{-\delta} \leq a$. Here,  $a$ is the constant
  in \eqref{eq:nondegeneracy}, and $\delta,C_0$ are as in
  \eqref{eq:smoothness}. We further require
  $t= \frac{N C_2}{\sqrt{n}}$, where
  $ \frac{1}{\sqrt{n}} \leq C_2 \leq 1$ and $N$ is an integer. The
 exact choice of the constant $C_2$ will be made clear below. Let $x_0,\,y_0$ be
  two points satisfying $x_0-y_0=tr\sqrt{n} u_0$, $r>0$, and consider
  the cubes $Q(x_0,r)$, $Q(y_0,r)$. Given any point $ x \in Q(x_0,r)$
  we can write $x= x_0 + h$, where $|h| < r \sqrt{n}$. Similarly,
  given $y \in Q(y_0,r)$, $ y= y_0+k$ where $|k| < r \sqrt{n}$.  We
  claim that for such $x$ and $y$,
\begin{equation}\label{eq:necessity measures 1}
| K(x,y)-K(x_0,y_0)| \leq \frac{1}{2} | K(x_0,y_0)|.
\end{equation}
To prove this we will apply \eqref{eq:smoothness} which is possible since $|h|< r\sqrt{n} \leq \frac{1}{2}|x_0-y_0|$, and
$$|x_0+h-y_0| \geq |x_0-y_0| -|h| \geq tr\sqrt{n} - r \sqrt{n} \geq \frac{t}{2} r \sqrt{n} \geq 2 |k|.$$
Thus, we can estimate as follows:
\begin{align} \label{eq:necessity measures 2}
|K(x,y)&- K(x_0,y_0)|\\
& \leq |K(x_0+h, y_0+k)-K(x_0+h,y_0)|+|K(x_0+h,y_0)-K(x_0,y_0)|\nonumber \\
& \leq  \frac{C_0|k|^\delta}{|x_0+h-y_0|^{n+\delta}} + \frac{C_0|h|^\delta}{|x_0-y_0|^{n+\delta}}\nonumber \\
&= I_1 + I_2 \nonumber.
\end{align}
We can bound $I_2$ immediately:
$$ I_2 \leq \frac{C_0(r \sqrt{n})^\delta}{(tr\sqrt{n})^\delta |x_0-y_0|^n} = C_0 \frac{ t^{-\delta}}{|x_0-y_0|^n}.$$
To estimate $I_1$, note that 
$$|x_0+h-y_0| \geq \frac{t}{2} r\sqrt{n}= \frac{1}{2}|x_0-y_0|.$$
Hence,
$$ I_1 \leq \frac{ C_0 2^{n+\delta}( r\sqrt{n})^\delta}{(tr\sqrt{n})^\delta|x_0-y_0|^n}= C_0\frac{ 2^{n+\delta}t^{-\delta}}{|x_0-y_0|^n }.$$
If we combine these estimates, by our choice of $t$ and \eqref{eq:nondegeneracy} we have
$$I_1+I_2 \leq \frac{a}{2} \frac{1}{|x_0-y_0|^n} \leq \frac{1}{2}|K(x_0,y_0)|,$$
which proves \eqref{eq:necessity measures 1}. 

It now follows that for any $ x\in Q(x_0,r)$, $ y\in Q(y_0,r)$, the
kernel $K(x,y)$ always has the same sign. Therefore, if we fix a
non-negative function $f$ with $\supp(f)\subset Q(x_0,r) $, then

\begin{align}\label{eq:necessity measures 4}
|Tf(y)|&= \left\lvert \int_{Q(x_0,r)} K(x,y) f(x) \; dx  \right\rvert \\
&= \int_{Q(x_0,r)} |K(x,y)| f(x) \; dx  \nonumber\\
& \geq \int_{Q(x_0,r)} |K(x_0,y_0)|f(x) \; dx - \int_{Q(x_0,r)} |K(x,y) -K(x_0,y_0)|f(x) \; dx; \nonumber\\
\intertext{again by  \eqref{eq:necessity measures 1}  and  \eqref{eq:nondegeneracy},\nonumber}
& \geq \frac{1}{2}|K(x_0,y_0)| \int_{Q(x_0,r)} f(x) \; dx \nonumber\\
& \geq \frac{a}{2|x_0-y_0|^n} \int_{Q(x_0,r)} f(x) \; dx \nonumber \\
& \geq  \frac{a}{2(tr \sqrt{n})^n} \int_{Q(x_0,r)} f(x) \; dx \nonumber\\
&=  c(a,t,n) \avgint_{Q(x_0,r)} f(x) \; dx.\nonumber
\end{align} 

Given this inequality and the assumption that $T$ satisfies a weak $(p,p)$ inequality  we have for any $0<\lambda< c(a,t,n) \avgint_{Q(x_0,r)} f \; dx$:
$$\mu(Q(y_0,r)) \leq  \mu( \{x:|Tf(y)| > \lambda \} \leq \frac{C}{\lambda^p} \int_{Q(x_0,r)} |f|^p \; d\nu.$$ 
If we take the supremum over all such $\lambda$, we get
\begin{equation}\label{eq:necessity measures 3}
\mu(Q(y_0,r)) \left(\avgint_{Q(x_0,r)} f \; dx \right)^p \leq c(a,t,n,p) \int_{Q(x_0,r)} |f|^p \; d\nu.
\end{equation}

Now fix a value of $C_2$, depending only on $u_0$, so that starting
from $Q(y_0,r)$ we can form a chain of adjacent cubes $Q(x_j,r)$,
$j=1, \dots N$ in the direction $-u_0$ such that $x_1=y_0$ and
$x_N=x_0$. Each $Q(x_j,r)$ satisfies
$ \mu(Q(x_{j+1},r)) \leq C_\mu \mu(Q(x_j ,r))$, where where $C_\mu$ is
the directional doubling constant from~\eqref{eq:directionaldoubling}.  The number of cubes $N$,
lying between $Q(y_0, r)$ and $Q(x_0,r)$ depends only on $t$ and $n$.
Thus, there exists constant $ C=C(C_\mu,t,n)$ such that
$ \mu(Q(x_0,r)) \leq C \mu(Q(y_0,r))$. Hence,
\begin{align*}
\mu(Q(x_0,r))\left(\avgint_{Q(x_0,r)} f \; dx \right)^p \leq C \mu(Q(y_0,r))  \left( \avgint_{Q(x_0,r)} f  \; dx \right)^p \leq C \int_{Q(x_0,r)} |f|^p \; d\nu.
\end{align*}
Since the resulting constant depends only on $C_1,\,p,\,t,\,n,\,a$ and not on
$Q(x_0,r)$ we have shown that the averaging operators
$A_Q : L^p( \nu) \rightarrow L^{p}(\mu)$ uniformly for all
$Q$. Therefore, by Theorem \ref{thm: necessity averaging operators
  measures} we get the desired conclusion.
\end{proof}

\section{Proofs of Theorems \ref{thm:necessity Tsigma} and \ref{thm:necessity Tailed Ap}}

Before proceeding with the proof of Theorem \ref{thm:necessity Tsigma} we first define the related averaging operator. 

\begin{defi}\label{def:averaging operator Asigma}
  Given a non-negative measure $\sigma$ and a cube $Q$, define the
  averaging operator $A_{Q,\sigma}$ acting on a function
  $f\in L^1_{\text{loc}}(\sigma)$ by
\begin{equation} \label{eq:Asigma}
A_{Q,\sigma} f(x) = \frac{1}{|Q|} \int_Q f(y) \; d\sigma(y) \chi_Q(x).
\end{equation}
\end{defi}

The following result characterizes the $L^p$ boundedness of $A_{Q,\sigma}$.
\begin{theorem}\label{thm: Neccessity Asigma}
Given  a cube $Q$, $1 \leq p < \infty$, and a pair of positive regular
Borel measures $(\mu, \sigma)$, suppose that $(\mu,\sigma)$ satisfy
the $A_p$ condition \eqref{eq:Ap condition for measures}.  
 Then for all $f \in L^p(\sigma)$,
\begin{equation*} 
 \|A_{Q,\sigma} f \|_{L^p(\mu)} \leq [\mu,\sigma]_{A_p}^{1/p} \|f\|_{L^p(\sigma)}.
\end{equation*}
Conversely given $1\leq p<\infty$, if $(\mu,\sigma)$ is a pair of
positive regular Borel measures such that for every cube $Q$,
\begin{equation}\label{eq:Neccessity Asigma 1}
\|A_{Q,\sigma} f \|_{L^p(\mu)} \leq K \|f\|_{L^p(\sigma)},
\end{equation}
then $(\mu,\sigma) \in A_p$. Moreover $[\mu,\sigma]_{A_p} \leq K^p .$
\end{theorem}
\begin{proof}
  We first prove necessity. Fix a cube $Q$ and let
  $1 \leq p < \infty$. If $\sigma(Q)=0$, then \eqref{eq:Ap condition
    for measures} is immediate. If $\sigma(Q)>0$ let $f=\chi_Q$ in
  \eqref{eq:Neccessity Asigma 1}. Then we obtain
  $$ \mu(Q) \left( \frac{\sigma(Q)}{|Q|} \right)^p \leq K^p \sigma(Q).$$
 Dividing by $\sigma(Q)$ and taking the supremum over all cubes $Q$ we have $(\mu,\sigma) \in A_p$ and $[\mu,\sigma]_{A_p} \leq K^p$. The proof of sufficiency is similar to Theorem \ref{thm:Ap neccessity averaging operators} so we omit the details.
\end{proof}

We can now prove Theorems~\ref{thm:necessity Tsigma}
and~\ref{thm:necessity Tailed Ap}.

\begin{proof}[Proof of Theorem \ref{thm:necessity Tsigma}]
  The proof is a straightforward modification of
  the proof of Theorem \ref{thm:necessity measures}.  Fix the cubes
  $Q(x_0,r)$ and $Q(y_0,r)$ as before.    Then with the same notation
  as before, we have that if $x\in Q(x_0,r)$ and $y\in Q(y_0,r)$,
  \[ |x-y|= |x_0-y_0+h-k| < tr\sqrt{n} +2r\sqrt{n} < 2tr\sqrt{n}. \]
  Similarly, we have $|x-y|> \tfrac{1}{2}tr\sqrt{n}$.  Therefore, if
  we choose $0<\epsilon < \tfrac{1}{4}tr\sqrt{n}$ and $R>4tr\sqrt{n}$,
  we have that the kernel $K_{\epsilon,R}(x,y)=K(x,y)$ and so
  satisfies the non-degeneracy condition~\eqref{eq:nondegeneracy} with
  a uniform constant.  We also have that it satisfies the standard
  estimates~\eqref{eq:size} and~\eqref{eq:smoothness}.

  For
  simplicity, we now  write $T_\sigma$ instead of
  $T_{\sigma}^{\epsilon,R}$ and $K$ for $K_{\epsilon,R}$. 
 If we repeat the previous argument, we have that $K$ satisfies
 the estimate \eqref{eq:necessity measures
    1}. We can then repeat the proof of \eqref{eq:necessity
    measures 4}, using the fact that $T_\sigma$ satisfies the weak
  $(p,p)$ inequality with uniform constant, to get
\begin{equation*}
|T_\sigma f(y) | \geq \frac{c(a,t,n)}{|Q(x_0,r)|} \int_{Q(x_0,r)} f(x) \; d\sigma(x).
\end{equation*}
Given this inequality we continue to argue as we did in the proof of
Theorem \ref{thm:necessity measures} to get that the averaging
operator $A_{Q,\sigma}=A_{Q(x_0,r),\sigma}$ satisfies
$A_{Q,\sigma}: L^p(\sigma) \rightarrow L^p(\mu)$.  This estimate holds
for every cube $Q(x_0,r)$ with constants independent of $\epsilon$ and
$R$, and so
$(\mu,\sigma) \in A_p$ by Theorem \ref{thm: Neccessity Asigma}.
\end{proof}

\begin{proof}[Proof of Theorem \ref{thm:necessity Tailed Ap}]
  We adapt the proof of Theorem \ref{thm:necessity measures},
  exchanging the roles of $x_0$ and $y_0$. Fix a cube
  $Q(y_0,r)$. Choose $t \geq 4$ as in the proof of Theorem
  \ref{thm:necessity measures}. Rather than considering the cube
  $Q(x_0,r)$ we replace it with a ball.  Fix $S>r$ and for each $r\leq
  s \leq S$ let
  $B_s= B(x_s, s \sqrt {n})$, where $x_s=y_0+ts\sqrt{n}u_0$.
If we now argue as we did in the proof of Theorem~\ref{thm:necessity
  Tsigma}, if we fix $R>4tS\sqrt{n}$ and
$\epsilon<\tfrac{1}{4}tr\sqrt{n}$, then for $y\in Q(y_0,r)$ and $x\in
B_s$, $K_{\epsilon,R}$ satisfies the non-degeneracy
  condition~\eqref{eq:nondegeneracy} with a uniform constant.  We also
  have that it satisfies the standard estimates~\eqref{eq:size}
  and~\eqref{eq:smoothness}.  Again, we will write $T_\sigma$ for
  $T_{\sigma}^{\epsilon,R}$ and $K$ for $K_{\epsilon,R}$. 

 We can now argue as follows:   for
  all $y \in Q(y_0,r)$, $y=y_0 + k$ where $|k| \leq r \sqrt{n}$ and
  for $x\in B_s$, $x= x_s +h$ where $|h| \leq s\sqrt{n}$. As in the
  proof of Theorem \ref{thm:necessity measures} we have
  $|h| \leq s\sqrt{n} \leq \frac{1}{2}|x_s-y_0|$, and
$$|x_s+h-y_0| \geq |x_s -y_0| - |h| \geq ts \sqrt{n}- s\sqrt{n} \geq \frac{t}{2} s \sqrt{n} \geq \frac{t}{2} r \sqrt{n} \geq 2|k|.$$
We can now apply \eqref{eq:smoothness} as in estimate
\eqref{eq:necessity measures 2} to get that for $y \in Q_r$ and $ x \in B_s$,
\begin{equation}
|K(x,y)-K(x_s,y_0)| \leq \frac{1}{2} |K(x_s,y_0)|.
\end{equation}
 This implies that for any $y \in Q(y_0,r)$ and $x\in B_s$ $K(x,y)$ always has the same sign. Moreover, we have
\begin{equation}
|K(x,y)| \geq \frac{1}{2} |K(x_s, y_0)|.
\end{equation}
Therefore,
\begin{multline}
|K(x,y)| \geq  \frac{1}{2}|K(x_s,y_0)| \geq \frac{a}{2} \frac{1}{|x_s-y_0|^n} 
 \geq \frac{a}{2} \frac{1}{(|x-x_s|+|x-y_0|)^n} \\
\geq c(a,n) \frac{1}{(|x-y_0|)^n} 
\geq c(a,n) \frac{1}{(|x-y_0|+r)^n};
\end{multline}
the second to last inequality follows since $|x-x_s| \leq s \sqrt{n} \leq ts \sqrt{n}= |x-y_0|$. 

 Define the truncated cone
$$C_r = \bigcup_{s \geq r} B_s .$$
Notice $C_r$ has a central axis of $y_0 + su_0$, for $s \geq r$.
For $S>0$ let $$f_{r,S}(x)= \left(\frac{1}{|x-y_0| + r)} \right)^{n(p'-1)} \chi_{C_r \cap B(y_0,S)}.$$

Then for all $ y \in Q(x_0,r)$ we have
\begin{multline}\label{eq: necessity Tailed Ap 1}
|T_\sigma f_{r,S}(y)| = \int_{C_r \cap B(y_0,S)} |K(x,y)|
\left(\frac{1}{(|x-y_0| + r )}\right)^{n(p'-1)} \; d\sigma(x) \\
\geq c(a,n) \int_{C_r \cap B(y_0,S)}
\left( \frac{1}{(|x-y_0|+ r)^{p'} } \right)^n d \sigma(x).
\end{multline}

We have that $T_\sigma$ satisfies the weak $(p,p)$ inequality with
uniform constant, so we can argue as we did to derive
\eqref{eq:necessity measures 3} in the proof of Theorem
\ref{thm:necessity measures} to get
\begin{multline*}
  \mu(Q(y_0,r)) \left( \int_{C_r \cap B(y_0,S)}
    \left( \; \frac{1}{( |x-y_0|+ r)^{p'} } \right)^{n} d\sigma(x)
  \right)^p \\
  \leq C \int_{C_r\cap B(y_0,S)}
  \left(\frac{1}{|x-y_0| +r } \right)^{np(p'-1)} d\sigma(x).
\end{multline*}
Since $p(p'-1)=p'$ we have 
\begin{equation*}
\frac{\mu(Q(y_0,r))}{|Q(y_0,r)|}  \left( \int_{C_r \cap B(y_0,S)} \left(\frac{r^{p'-1}}{(|x-y_0|+r)^{p'}} \right)^n d\sigma(x) \right)^{p-1} \leq C.
\end{equation*}
Since the constant $C$ is independent of $\epsilon$ and $R$, and so of
$S$,  we can
take the limit as  $S\rightarrow \infty$, and by the monotone
convergence theorem we get
\begin{equation}\label{eq: necessity Tailed Ap 2}
\frac{\mu(Q(y_0,r))}{|Q(y_0,r)|}  \left( \int_{C_r} \left(\frac{r^{p'-1}}{(|x-y_0|+r)^{p'}} \right)^n d\sigma(x) \right)^{p-1} \leq C.
\end{equation}

\medskip

We will now extend inequality \eqref{eq: necessity Tailed Ap 2} to all
of $\mathbb{R}^n$. Let
$$A_k= B(y_0, 2^{k+1} tr \sqrt{n}) \setminus B(y_0, 2^k tr
\sqrt{n}).$$ Consider the ball $B(x_k, 2^{k+2}tr \sqrt{n})$, where
 $$ x_k= y_0 + (\frac{2^{k+1} +2^{k}}{2})tr \sqrt{n} u_0=  y_0+ \tfrac{3}{8}(2^{k+2})tr \sqrt{n} u_0.$$
This is the ball of radius $2^{k+2}tr \sqrt{n}$ centered at the midpoint of the portion of the central axis of $C_r$  that lies inside $A_k$. We claim $A_k \subset B(x_k,2^{k+2} tr \sqrt{n})$. To see this, fix $x\in A_k$; then  
$$|x-x_k| \leq |x-y_0|+|x_k -y_0| \leq 2^{k+1}tr\sqrt{n}+ \tfrac{3}{8} (2^{k+2})tr\sqrt{n}  \leq 2^{k+2} tr\sqrt{n}.$$
Since the ball $B(x_k, \frac{3}{8} (2^{k+2}) r \sqrt{n})$ is one of the balls $B_s$ that defines $C_r$, it is immediate that 
$$ \bigcup^\infty_{k=0} B(x_k, \tfrac{3}{8} (2^{k+2}) r\sqrt{n}) \subset C_r. $$
Since $\sigma$ is doubling there exists a constant $C=C(t,n, \sigma)$ such that 
$$ \sigma(B(x_k, 2^{k+2} tr \sqrt{n})) \leq C \sigma \left(B(x_k, \tfrac{3}{8} (2^{k+2})r \sqrt{n}) \right). $$

Hence, we can estimate as follows:
\begin{align*}
&\frac{\mu(Q(y_0,r))}{|Q(y_0,r)|} \left( \int_{\mathbb{R}^n \setminus B(y_0 ,tr\sqrt{n})} \left( \frac{r^{p'-1}}{(|x-y_0| + r)^{p'}} \right)^n d\sigma(x) \right)^{p-1} \\
& \quad= \frac{\mu(Q(y_0,r))}{|Q(y_0,r)|}   \left ( \sum^\infty_{k=0}\int_{A_k}  \left( \frac{r^{p'-1}}{(|x-y_0| + r)^{p'}} \right)^n d\sigma(x) \right)^{p-1}  \\
& \quad\leq \frac{\mu(Q(y_0,r))}{|Q(y_0,r)|}  \left(\sum^\infty_{k=0}    \left( \frac{r^{p'-1}}{(2^ktr\sqrt{n} + r)^{p'}} \right)^n \sigma(A_k) \right)^{p-1} \\
& \quad\leq   \frac{\mu(Q(y_0,r))}{|Q(y_0,r)|}  \left( \sum_{k=0}^\infty\left(\frac{r^{p'-1}}{(2^ktr\sqrt{n} + r)^{p'}} \right)^n \sigma(B(x_k, 2^{k+2} tr\sqrt{n} )) \right)^{p-1} \\
& \quad \leq C  \frac{\mu(Q(y_0,r))}{|Q(y_0,r)|}  \left( \sum_{k=0}^\infty\left(\frac{r^{p'-1}}{(2^ktr\sqrt{n} + r)^{p'}} \right)^n \sigma(B(x_k, \tfrac{3}{8}(2^{k+2}) r\sqrt{n} )) \right)^{p-1} \\
& \quad\leq C \frac{\mu(Q(y_0,r))}{|Q(y_0,r)|}  \left( \sum_{k=0}^\infty\int_{B(x_k, \frac{3}{8}(2^{k+2})r\sqrt{n})}\left(\frac{r^{p'-1}}{(|x-y_0| + r)^{p'}} \right)^n d\sigma(x) \right)^{p-1} \\
& \quad \leq C \frac{\mu(Q(y_0,r))}{|Q(y_0,r)|}  \left( \int_{C_r} \left(\frac{r^{p'-1}}{(|x-y_0|+r)^{p'}} \right)^n d\sigma(x) \right)^{p-1} \\
& \quad \leq C.
\end{align*}
The third to last inequality holds since $2^{k}tr\sqrt{n} \geq
\frac{1}{2} |x-y_0|$ for any $x \in B(x_k,
\tfrac{3}{8}(2^{k+2})r\sqrt{n})$.

By Remark \ref{remark: Ap balls cubes}, we can apply the result of
Theorem \ref{thm:necessity Tsigma} to the ball $B(y_0, tr\sqrt{n})$ to
get
\begin{align*}
 & \left(\frac{\mu(Q(y_0,r))}{|Q(y_0,r)|}\right)
  \left(\int_{B(y_0, tr\sqrt{n})}
  \left(  \frac{r^{p'-1}}{(|x-y_0|+r)^{p'}} 
  \right)^n d\sigma(x) \right)^{p-1}\\
& \qquad \qquad \qquad  \leq \left(\frac{\mu(Q(y_0,r))}{|Q(y_0,r)|}\right) \left(
   \frac{\sigma(B(y_0, tr\sqrt{n}))}{|B(y_0,tr
     \sqrt{n})|}\right)^{p-1} \\
& \qquad \qquad \qquad  \leq \left(\frac{ \mu(B(y_0,
                                   tr\sqrt{n}))}{|B(y_0,tr\sqrt{n})|}
                                   \right) \left( \frac{\sigma(B(y_0,
                                   tr\sqrt{n}))}{|B(y_0,tr
                                   \sqrt{n})|}\right)^{p-1} \\
& \qquad \qquad \qquad   \leq C.
\end{align*}
If we 
combine this inequality with the previous estimate, we get 
 $$\left(\frac{\mu(Q(y_0,r))}{|Q(y_0,r)|}\right)
 \left(\int_{\mathbb{R}^n} \left( \frac{r^{p'-1}}{(|x-y_0|+r)^{p'}}
   \right)^n d\sigma(x) \right)^{p-1} \leq C. $$
 Since this holds for every cube $Q(y_0,r)$, it follows that
 $(\mu,\sigma)$ satisfy the $PA_p$ condition.
\end{proof}

\section{Strong $(1,1)$ Inequalities}

For the proof of Theorem \ref{thm:strong 1-1} we first give some
preliminary lemmas.

\begin{lemma}\label{lem:essential infimum}
 Let $v$ be a measurable function.  Then for a.e. $x \in \mathbb{R}^n$,
\begin{equation}\label{eq:essential infimum 1}
 \lim_{r \rightarrow 0^+} [\essinf_{ y \in Q(x,r)} v(y)] \leq v(x). 
 \end{equation}
\end{lemma}

The proof of Lemma \ref{lem:essential infimum} is implicit in
\cite[~Theorem 4]{MR0417671} in one dimension; the proof is the same
in higher dimensions.

\begin{defi}
A family $\{E_r\}_{r>0}$ of Borel subsets of $\mathbb{R}^n$ is said to shrink nicely to $x \in \mathbb{R}^n$ if 
$$ E_r \subset B(x,r) \quad \text{for each}\; r,$$
and there exists a constant $\alpha$ independent of $r$ such that 
$$|E_r|> \alpha | B(x,r)|.$$ 
\end{defi}
\begin{lemma}\label{lem:differentiation theorem}
Let $\mu$ be a  regular Borel measure on $\mathbb{R}^n$, and let $d\mu=d\mu_s +u\; dx$ be its Lebesgue Radon-Nikodym decomposition. Then for a.e. $x \in \mathbb{R}^n$, 
$$ \lim_{r \rightarrow 0} \frac{ \mu(E_r)}{|E_r|} = u(x).$$  
\end{lemma}
The proof of Lemma \ref{lem:differentiation theorem} can be found in \cite[Theorem~3.22, p.99]{MR1681462}. 

\begin{proof}[Proof of Theorem \ref{thm:strong 1-1}.]
  First suppose that  the measure $\nu$ is singular with respect to Lebesgue
  measure. As in the proof of Theorem \ref{thm:necessity Tailed Ap}
  fix a cube $Q(y_0,r)$ and define the truncated cone $C_r$.  Let $f$
  be a non negative function with $\supp(f) \subset Q(y_0,r)$. Then, if
  we estimate as in the proof of Theorem \ref{thm:necessity Tailed Ap}
  to get \eqref{eq: necessity Tailed Ap 1}, we have for all
  $x \in C_r$,
$$|Tf(x)|  \geq c(a,n) \int_{Q(y_0,r)} \frac{f(y)}{(r+|x-y_0|)^n} \; dy.$$
By assumption $T: L^1(\nu) \rightarrow L^1(\mu)$, so we have that 
\begin{align}\label{eq:strong 1-1 1}
\int_{Q(y_0,r)} f(x)\; d\nu(x)  &\geq  c \int_{\mathbb{R}^n} |Tf(x)| \; d\mu(x) \\ \nonumber
& \geq  c \int_{C_r} \int_{Q(y_0,r)} \frac{f(y)}{(r+|x-y_0|)^n}  dy \; d\mu(x) \\  \nonumber
& =  c \int_{Q(y_0,r)} f(y) \; dy \int_{C_r} \frac{1}{(r+|x-y_0|)^n} \; d\mu(x).
\end{align}

If $\mu \neq 0$, since $\mu$ is a Borel measure, there exists a ball
$B$ such that $\mu(B) >0$. Fix a point $y_0$ and $r>0$ such that
$B \subset C_r$.  Let $f= \chi_{Q(y_0, r) \setminus \supp(\nu)}$ in
inequality \eqref{eq:strong 1-1 1}; then the left-hand side equals
$0$. Since $\nu$ is singular with respect to Lebesgue measure,
$|\supp(\nu)|=0$, so the first term on the right-hand side is
positive. Since the integrand in the second term on the right-hand
side is bounded away from $0$, the second term is positive unless
$\mu(C_r)=0$, a contradiction. Hence, $\mu =0$.

\medskip

Now let $\mu$ be a regular measure with Lebesgue decomposition
$d\mu= d\mu_s + u \; dx$, where $u \not\equiv 0$, and suppose
$d\nu=d\nu_s + vdx$, where $v$ is a non-negative function such that
$v(x)< \infty$ a.e.  Fix a point $y_0$ such that
$ 0< u(y_0) < \infty$. We can further assume that $y_0$ is a Lebesgue
point for $\mu$ in the sense of Lemma \ref{lem:differentiation
  theorem}, and that the conclusion of Lemma \ref{lem:essential
  infimum} holds for the function $v$ at $y_0$. Let
$a= \essinf_{ x \in Q(y_0,r) } v(x)$. Given $\epsilon>0$, let
$E= \{ x \in Q(y_0,r): v(x) < a + \epsilon \}$,
$A=E \setminus \supp(\nu_s)$, and set $ f= |A|^{-1} \chi_{A}$ in
inequality \eqref{eq:strong 1-1 1}.  By the definition of the
essential infimum, $|E|>0$, and since $|\supp(\nu_s)|=0$, we have that
$|A|>0$.  Thus,
$$ \int_{C_r} \frac{1}{(r+|x-x_0|)^n} \; d\mu(x)
\leq C\frac{\nu(A)}{|A|} \leq C\frac{v(A)}{|A|} 
\leq C(a+ \epsilon)
  = C [ \essinf_{x \in Q(y_0,r)} v(x) + \epsilon].$$
  Since $\epsilon>0$ was arbitrary, this inequality holds with
  $\epsilon=0$.    As $r\rightarrow 0$, $C_r$ converges to the cone $C_0$ with central
  axis $y_0+ ts\sqrt{n} u_0, s\geq 0$. Therefore, by the monotone
  convergence theorem and Lemma \ref{lem:essential infimum} we have
\begin{equation}\label{eq:strong 1-1 3}
\int_{C_0} \frac{1}{|x-y_0|^n} \; d\mu(x) \leq C v(y_0).
\end{equation}   
Let $B_j=B(y_0, 2^{-j}),  A_j=(C_0\cap B_j) \setminus B_{j+1}$. Since
$C_0$ has constant aperture, there exists $0<\alpha<1$ such that
$|A_j|=\alpha|B_j|$, so the collection $\{A_j\}$ shrinks nicely to $y_0$. Then we have that
\begin{equation}\label{eq:strong 1-1 4}
    \lim_{j\rightarrow \infty} \frac{\mu(A_j)}{|B_j|} = \alpha u(y_0). 
\end{equation}
Fix $j_0$ such that for all $j \geq j_0$ we have $\frac{\mu(A_j)}{|B_j|} \geq \frac{\alpha}{2} u(y_0)$.  Hence,
\begin{multline*}
v(y_0) \geq c \sum_{j \geq j_0} \int_{A_j} \frac{1}{|x-y_0|^n} d\mu(x)\\
\geq c\sum_{j \geq j_0} 2^{nj} \mu(A_j)
\geq c\sum_{j \geq j_0} \frac{\mu(A_j)}{|B_j|}  
 \geq c \sum_{j \geq j_0} u(y_0)= \infty.  
\end{multline*}
Let $E= \{ x: 0<u(x)< \infty \}$ which has positive measure since
$ u\not\equiv 0$. Then we have $v(x)= \infty$ for a.e $x\in E$. This
contradicts the fact that $ v(x)< \infty$ a.e.

\medskip

Finally, suppose $\mu$ is a  regular  measure that is
singular with respect to Lebesgue measure, and is directionaly doubling
in the direction $u_0$, and $\nu$ is a positive regular Borel
measure. If $T: L^1(\nu) \rightarrow L^1(\mu)$, then it satisfies a
weak $(1,1)$ inequality, and so by Theorem
\ref{thm:necessity measures} $\mu$ is absolutely continuous; a
contradiction.
\end{proof}

\bibliographystyle{plain}
\bibliography{necessity}

\end{document}